\documentclass{article}
\usepackage{tikz}
\usepackage{multirow}
\usepackage{pstricks,color,pst-node,pst-tree}
\usepackage{amssymb,latexsym}
\usepackage{amsfonts,amsmath}
\topskip  \usepackage{tikz}
\usepackage{multirow}
\usepackage{pstricks,color,pst-node,pst-tree}
\usepackage{amssymb,latexsym}
\usepackage{amsfonts,amsmath}
\topskip  \newenvironment{proof}[1][Proof]{\paragraph*{#1}}{\hspace*{\fill}$\Box$\bigskip}
\newtheorem{theorem}{Theorem}
\newtheorem{proposition}[theorem]{Proposition}
\newtheorem{corollary}[theorem]{Corollary}

\newtheorem{lemma}[theorem]{Lemma}

\def\sumrecord{{\rm srec}}
\def\sumweightedrecord{{\rm swrec}}
\begin{document}
\title{Sum of weighted records in set partitions}

\author{Walaa Asakly \\
Department of Computer Science, University of Haifa, 3498838 Haifa, Israel\\
{\tt walaa\_asakly@hotmail.com}}
\date{\small }
\maketitle

\begin{abstract}
The purpose of this paper is to find an explicit formula and asymptotic estimate
for the total number of sum of weighted records over set partitions of $[n]$
 in terms of Bell numbers. For that we study the generating function
for the number of set partitions of $[n]$ according to the statistic
sum of weighted records.
\medskip

\end{abstract}
\section{Introduction}
Let $\sigma_i$ be an element in the permutation $\sigma=\sigma_1 \sigma_2 \cdots \sigma_{\ell}$,
we say that $\sigma_i$ is a {\em record} of {\em position} $i$ if $\sigma_i>\sigma_j$ for all $j=1, 2, \cdots, i-1$.
The study of records in permutations interested R\'enyi \cite{AR}.
 More recently another statistic which depends on records have been studied by Kortchemski \cite{I}
who defined the statistic {\em srec}, where $\sumrecord(\sigma)$ defined as the sum of the positions of all records of $\sigma$.
 For example, permutation $\pi=12534$ has 3 records, 1, 2, 5 and
$\sumrecord(\pi)=1+2+3=6$. For relevant papers about records you can see for example \cite{N} and \cite{AT}.
In this Paper we want to focus on partitions of a set. Recall that a {\em partition} $\Pi$ of set $[n]$ of size
$k$ ({\em a partition of $[n]$ with exactly $k$ blocks}) is a collection $\{B_1,B_2,\ldots,B_k\}$,
 where $\emptyset \neq B_i\subseteq [n]$ for all $i$ and $B_i\bigcap B_j=\emptyset$ for $i\neq j$, such
 that $\bigcup_{i=1}^{k}B_i=[n].$
 The elements $B_i$ are called {\em blocks}, and we use the assumption that $B_1, B_2, \cdots, B_k$
 are listed in increasing order of their minimal elements, that is, $minB_1 <
minB_2 < \cdots< min B_k$. The set of all partitions of
$[n]$ with exactly $k$ blocks is denoted by $P_{n,k}$ and $|P_{n,k}|=S_{n,k}$, which is
known as the {\em Stirling numbers of the second kind} \cite{S}.
And  the set of all partitions of $[n]$ is denoted by $P_n$ and $|P_n|=\sum_{k=1}^{n}S_{n,k}=B_n$,
 which is the {\em n-th Bell number} \cite{S}.
 Any partition $\Pi$ can be written as
$\pi_1\pi_2\cdots\pi_n$, where $i\in B_{\pi_i}$ for all $i$, and
this form is called the {\em canonical sequential form}. For
example $\Pi=\{\{12\},\{3\},\{4\}\}$ is a partition of $[4]$, the
canonical sequential form is $\pi=1123$.
 For more details about set partitions we suggest Mansour's book \cite{TM}. 
The important results about records, obtained by Knopfmacher, Mansour  and  Wagner \cite{ATS}
which state the asymptotic mean value and variance for the number, and for the sum of positions,
of record in all partitions of $[n]$ are central to my study.  
 In this paper, we define a new statistic {\em swrec}, where $\sumweightedrecord(\pi)$
 is the sum of the position of a record in $\pi$ multiplied by the value
of the record over all the records in $P_n$. We will study this statistic from
the point of view of canonical sequential form. For instance, if $\pi=121132$ the
$\sumweightedrecord(\pi)=1\cdot1+2\cdot2+3\cdot5=20$.

\section {Main Results}
\subsection{The ordinary generating function for the number of set partitions according to the statistic $\sumweightedrecord$}
Let $P_k(x,q)$ be the generating function for the number of partitions of
$[n]$ with exactly $k$ blocks according to the statistic $\sumweightedrecord$, that is
$$P_k(x,q)=\sum_{n\geq k}\sum_{\pi\in P_{n,k}}x^nq^{\sumweightedrecord(\pi)}.$$

\begin{theorem}\label{th1}
The generating function for the number of partitions of $[n]$ with exactly $k$ blocks
according to the statistic $\sumweightedrecord$ is given by

\begin{align}\label{eq1}
P_k(x,q)=\prod_{i=1}^{k}\frac{xq^iq^{(k+1-i)(k-i)}}{1-ix\prod_{j=i+1}^{k}q^j}.
\end{align}
\end{theorem}

\begin{proof}
As we know, a set partition of $[n]$ with exactly $k$ blocks can be presented as canonical sequential form:
$$\pi=1\pi^{(1)}2\pi^{(2)}\cdots k\pi^{(k)}$$
for some $k$, where $\pi^{(j)}$ denotes an arbitrary word over an
alphabet $[j]$ including the empty word. Thus,
the contribution of $\pi=1\pi^{(1)}2\pi^{(2)}\cdots k$ to the generating function
$P_k(x,q)$ is $xq^kP_{k-1}(xq^k,q)$ and the contribution of $\pi^{(k)}$
to the generating function $P_k(x,q)$ is $\frac{1}{1-kx}$.
 Therefore, the corresponding generating function satisfies
\begin{align*}
P_k(x,q)=\frac{xq^k}{1-kx}P_{k-1}(xq^k,q)
\end{align*}
By using induction on $k$ together with the initial condition $P_1(x,q)=\frac{xq}{1-x}$
we obtain the required result.

\end{proof}

\subsection{Exact and asymptotic expression for $\sum\limits_{\pi\in P_n}\sumweightedrecord(\pi)$}
In this section, we aim to prove that the total number of the $\sumweightedrecord$
over all partitions of $[n]$ is
$$\frac{3}{4}(B_{n+3}-B_{n+2})-(n+\frac{7}{4})B_{n+1}-\frac{1}{2}(n+1)B_n.$$
And we want to show that asymptotically the total number of the $\sumweightedrecord$
over all partitions of $[n]$ is
$$B_n\frac{n^3}{r^3}\left(1+\frac{r}{n}\right)\left(1+O(\frac{\log n}{n})\right),$$
where $r$ is the positive root of $re^r=n+1$.
\\For that we need to perform the following steps:
\\$\bullet$ Firstly, we find the partial derivative
of $P_k(x,q)$ with respect to $q$ and substitute $q=1$,
 that is $\frac{d}{dq}P_k(x,q)\mid_{q=1}$.
\\$\bullet$ Secondly, we pass from  $\frac{d}{dq}P_k(x,q)\mid_{q=1}$
 to $\frac{d}{dq}\widetilde{P}_k(x,u,q)\mid_{u=q=1}$,
 where $\widetilde{P}_k(x,u,q)$ is the exponential generating function
for the number of partitions of $[n]$ with exactly $k$ blocks
according to the statistic $\sumweightedrecord$.
\\$\bullet$ Finally, we derive the total number of $\sumweightedrecord$
over all partitions of $[n]$, and the asymptotic estimate for
the total number of $\sumweightedrecord$ over all partitions of $[n]$.
\begin{lemma}\label{L}
For all $k\geq
1$,
\begin{align}\label{eq2}
\frac{d}{dq}P_k(x,q)\mid_{q=1}=\frac{x^{k}(\binom{k+1}{2}+ 2\binom{k+1}{3})}{(1-x)\ldots (1-kx)}+\frac{x^{k+1}}{(1-x)\ldots (1-kx)}\sum_{i=1}^{k}{\frac{i(i+1+k)(k-i)}{2(1-ix)}}.
\end{align}
\end{lemma}
\begin{proof}
By differentiating (\ref{eq1}) with respect to $q$, we obtain

\begin{equation}\label{eq3}
\frac{d}{dq}P_k(x,q)\mid_{q=1}=P_k(x,1)\sum_{i=1}^{k}{\lim_{q \rightarrow 1}\left(\frac{\frac{d}{dq}L_i(q)}{L_i(q)}\right)},
\end{equation}
where $$L_i(q)=\frac{xq^iq^{(k+1-i)(k-i)}}{1-ix\prod_{j=i+1}^{k}q^j}.$$
We have
\begin{equation}\label{eq4}
\lim_{q\rightarrow 1}\frac{d}{dq}L_i(q)=\frac{x\ell A_{i,k}(x,1)-x\lim_{q\rightarrow 1}\frac{d}{dq}A_{i,k}(x,q)}{(A_{i,k}(x,1))^2}.
\end{equation}
Where $A_{i,k}(x,q)=1-ix\prod_{j=i+1}^{k}q^j$ and $\ell=(k+1-i)(k-i)+i$. By using the differentiation rules
we get  $\frac{d}{dq}A_{i,k}(x,q)=-ix\sum_{m=i+1}^{k}mq^{m-1}\prod_{\substack{j=i+1 \\ j\neq m}}^{k}q^j$.
Therefore,
  \begin{align*}
&\lim_{q\rightarrow 1}\frac{d}{dq}L_i(q)=\frac{x\left(2(i+(k+1-i)(k-i))(1-ix)+i(i+1+k)(k-i)x\right)}{2(1-ix)^2},
\end{align*}
which leads to
\begin{equation}\label{eq5}
\lim_{q \rightarrow 1}\left(\frac{\frac{d}{dq}L_i(q)}{L_i(q)}\right)=i+(k+1-i)(k-i)+\frac{i(i+1+k)(k-i)x}{2(1-ix)}.
\end{equation}

Hence, by substituting (\ref{eq5}) in (\ref{eq3}) we obtain
\begin{align*}
&\frac{d}{dq}P_k(x,q)\mid_{q=1}=\frac{x^k}{(1-x)\ldots(1-kx)}\sum_{i=1}^{k}\left(i+(k+1-i)(k-i)+\frac{i(i+1+k)(k-i)x}{2(1-ix)}\right)\\
&=\frac{x^{k}(\binom{k+1}{2}+ 2\binom{k+1}{3})}{(1-x)\ldots (1-kx)}+\frac{x^{k+1}}{(1-x)\ldots (1-kx)}\sum_{i=1}^{k}{\frac{i(i+1+k)(k-i)}{2(1-ix)}},
\end{align*}
as claimed.
\end{proof}

Now we need to find $[x^n]\frac{d}{dq}P_k(x,q)\mid_{q=1}$ to obtain the total number of $\sumweightedrecord$.
 We will study the exponential generating function instead of the ordinary generating function.
Let $\widetilde{P}_k(x,u,q)$ be the exponential generating function for the number of partitions of
$[n]$ with exactly $k$ blocks according to the statistic $\sumweightedrecord$, that is
$$\widetilde{P}_k(x,u,q)=\sum_{n\geq k}\sum_{\pi\in P_{n,k}}\frac{x^nu^kq^{\sumweightedrecord(\pi)}}{n!}.$$

\begin{theorem}\label{th2}
The partial derivative of $\widetilde{P}_k(x,u,q)$ with respect to $q$ at $u=q=1$ is given by,
\begin{align}
\frac{d}{dq}\widetilde{P}_k(x,u,q)\mid_{u=q=1}=e^{e^x-1}(\frac{3}{4}e^{3x}+\frac{3}{2}e^{2x}-\frac{7}{4}e^x-xe^{2x}-\frac{3}{2}xe^x-\frac{1}{2}).
\end{align}
\end{theorem}
\begin{proof}
In order to prove the above result we need the following proposition:
\begin{proposition}\label{LL}
The partial derivative $\frac{d}{dq}P_k(x,q)\mid_{\substack{x=y^{-1}\\q=1}}$ can be decomposed as
\begin{align}\label{eqq}
\sum_{m=1}^{k}\left(\frac{a_{k,m}}{(y-m)^2}+\frac{b_{k,m}}{y-m}\right),
\end{align}
 where $$a_{k,m}=\frac{(-1)^{k-m}m(1+k+m)(k-m)}{2(m-1)!(k-m)!},$$
and
$$b_{k,m}=\frac{(-1)^{k-m}\left(k^2(\frac{m}{4}+1)+k(\frac{m^2}{2}+\frac{3m}{4}+1)-(\frac{3m^2}{2}+m)\right)}{(m-1)!(k-m)!}.$$
\end{proposition}
\begin{proof}
  We rewrite (\ref{eq2}) as

\begin{align*}
&\frac{d}{dq}P_k(x,q)\mid_{q=1}=x^k\prod_{i=1}^{k}(1-ix)^{-1}\left(\binom{k+1}{2}+2\binom{k+1}{3}+\sum_{i=1}^{k}\frac{i(1+k+i)(k-i)x}{2(1-ix)}\right)\\
=&x^k\prod_{i=1}^{k}(1-ix)^{-1}\left(\frac{k(k+1)(2k+1)}{6}+\sum_{i=1}^{k}\frac{i(1+k+i)(k-i)x}{2(1-ix)}\right).
\end{align*}
By replacing $x^{-1}=y$ in the above equation we get

\begin{align}\label{eq4}
\prod_{i=1}^{k}(y-i)^{-1}\left(\frac{k(k+1)(2k+1)}{6}+\sum_{i=1}^{k}\frac{i(1+k+i)(k-i)}{2(y-i)}\right).
\end{align}
The above expression decomposed as

$$\sum_{m=1}^{k}\left(\frac{a_{k,m}}{(y-m)^2}+\frac{b_{k,m}}{y-m}\right).$$

In order to find the coefficients $a_{k,m}$ and $b_{k,m}$, we need to consider the expansion of (\ref{eq4}) at $y=m$,
 as follows:

\begin{align*}
&(y-m)^{-1}\prod_{\substack{i=1 \\ i\neq m}}^{k}(y-m+m-i)^{-1}\left(\frac{k(k+1)(2k+1)}{6}+\frac{m(1+k+m)(k-m)}{2(y-m)}+\sum_{\substack{i=1 \\ i\neq m}}^{k}\frac{i(1+k+i)(k-i)}{2(y-i)}\right)\\
&=(y-m)^{-1}\prod_{\substack{i=1 \\ i\neq m}}^{k}\left((m-i)^{-1}(1+\frac{y-m}{m-i})^{-1}\right)\cdot\\
&\left(\frac{k(k+1)(2k+1)}{6}+\frac{m(1+k+m)(k-m)}{2(y-m)}+\sum_{\substack{i=1 \\ i\neq m}}^{k}\frac{i(1+k+i)(k-i)}{2(y-i)}\right).
\end{align*}
Using Taylor series to expand  $(1+\frac{y-m}{m-i})^{-1}$ and $\frac{i(1+k+i)(k-i)}{2(y-i)}$ at $y=m$ we get
\begin{align*}
&(y-m)^{-1}\frac{(-1)^{k-m}}{(m-1)!(k-m)!}\prod_{\substack{i=1 \\ i\neq m}}^{k}\left(1-\frac{y-m}{m-i}+O((y-m)^2)\right)\\
&\cdot\left(\frac{k(k+1)(2k+1)}{6}+\frac{m(1+k+m)(k-m)}{2(y-m)}+\sum_{\substack{i=1 \\ i\neq m}}^{k}\frac{i(1+k+i)(k-i)}{2(m-i)}+O(y-m)\right),
\end{align*}
which is equivalent to 
\begin{align*}
&(y-m)^{-1}\frac{(-1)^{k-m}}{(m-1)!(k-m)!}\left(1-\sum_{\substack{i=1 \\ i\neq m}}^{k}\frac{y-m}{m-i}+O((y-m)^2)\right)\\
&\cdot\left(\frac{k(k+1)(2k+1)}{6}+\frac{m(1+k+m)(k-m)}{2(y-m)}+\sum_{\substack{i=1 \\ i\neq m}}^{k}\frac{i(1+k+i)(k-i)}{2(m-i)}+O(y-m)\right).
\end{align*}
We need to simplify the product, and consider the coefficients of $(y-m)^{-1}$ and $(y-m)^{-2}$ as follows:
\begin{align*}
&(y-m)^{-1}\frac{(-1)^{k-m}}{(m-1)!(k-m)!}\\
&\cdot\left(\frac{k(k+1)(2k+1)}{6}+\frac{m(1+k+m)(k-m)}{2(y-m)}+\sum_{\substack{i=1 \\ i\neq m}}^{k}\frac{i(1+k+i)(k-i)-m(1+k+m)(k-m)}{2(m-i)}+O(y-m)\right).
\end{align*}
By using Maple we compute the term in the summation, which hints
\begin{align*}
&(y-m)^{-1}\frac{(-1)^{k-m}}{(m-1)!(k-m)!}\\
&\cdot\left(\frac{m(1+k+m)(k-m)}{2(y-m)}+\frac{k(k+1)(2k+1)}{6}+\frac{1}{2}\sum_{\substack{i=1 \\ i\neq m}}^{k}(i^2+i+im-k^2-k+m+m^2)+O(y-m)\right)\\
&=(y-m)^{-1}\frac{(-1)^{k-m}}{(m-1)!(k-m)!}\\
&\cdot\left(\frac{m(1+k+m)(k-m)}{2(y-m)}+k^2(\frac{m}{4}+1)+k(\frac{m^2}{2}+\frac{3m}{4}+1)-(\frac{3m^2}{2}+m)+O(y-m)\right).
\end{align*}
Hence, by finding the coefficients of of $(y-m)^{-1}$ and $(y-m)^{-2}$ we complete the proof.
\end{proof}
\\Now we return to the proof of the theorem. We use (\ref{eqq}) for passing to exponential generating function, by substituting
$$\frac{e^{mx}-1}{m}=\sum_{\ell\geq0}\frac{m^\ell x^{\ell+1}}{(\ell+1)!}$$
in
$$\frac{1}{y-m}=\frac{x}{1-mx}=\sum_{\ell\geq 0}m^\ell x^{\ell+1}$$
and $\frac{e^{mx}(mx-1)+1}{m^2}$ in $\frac{1}{(y-m)^2}$. Moreover, by summing over all $k$ we obtain
the generating function

\begin{align*}
&\sum_{k\geq 1}u^k{\sum_{m=1}^{k}\frac{(-1)^{k-m}m(1+k+m)(k-m)}{2(m-1)!(k-m)!}\cdot \frac{e^{mx}(mx-1)+1}{m^2}}\\
&+\sum_{k\geq 1}u^k{\sum_{m=1}^{k}\frac{(-1)^{k-m}\left(k^2(\frac{m}{4}+1)+k(\frac{m^2}{2}+\frac{3m}{4}+1)-(\frac{3m^2}{2}+m)\right)}{(m-1)!(k-m)!}\cdot
\frac{e^{mx}-1}{m}}.
\end{align*}
We need to change the order of the summation as follows:
\begin{align*}
&\sum_{m\geq1}\frac{e^{mx}(mx-1)+1}{m!}\sum_{k\geq m}\frac{(-1)^{k-m}(1+k+m)(k-m)u^k}{2(k-m)!}\\
&+\sum_{m\geq1}\frac{e^{mx}-1}{m!}\sum_{k\geq m}\frac{(-1)^{k-m}\left(k^2(\frac{m}{4}+1)+k(\frac{m^2}{2}+\frac{3m}{4}+1)-(\frac{3m^2}{2}+m)\right)}{(k-m)!}u^k.
\end{align*}
By substituting $\ell=k-m$ and rewriting the above result we obtain the following form:
\begin{align*}
&\sum_{m\geq1}\frac{e^{mx}(mx-1)+1}{m!}\sum_{\ell\geq 0}\frac{(-1)^{\ell}(2m+\ell+1)\ell u^{m+\ell}}{2\ell!}\\
&+\sum_{m\geq1}\frac{e^{mx}-1}{m!}\sum_{\ell \geq 0}\frac{(-1)^{\ell}\left((m+\ell)^2(\frac{m}{4}+1)+(m+\ell)(\frac{m^2}{2}+\frac{3m}{4}+1)-(\frac{3m^2}{2}+m)\right)}{\ell!}u^{m+\ell}.
\end{align*}

By evaluating the previous terms in $u=1$ we complete the proof.
\end{proof}

\begin{theorem}\label{th}
The total number of $\sumweightedrecord$ taken over all set partitions of $[n]$, is given by
$$\frac{3}{4}(B_{n+3}-B_{n+2})-(n+\frac{7}{4})B_{n+1}-\frac{1}{2}(n+1)B_n.$$
\end{theorem}
\begin{proof}
In order to find the total number of $\sumweightedrecord$, we need to find an explicit formula for the coefficient of $x^n$
in the generating function $\frac{d}{dq}\widetilde{P}_k(x,u,q)\mid_{u=q=1}$. By Theorem \ref{th2}
$$\frac{d}{dq}\widetilde{P}_k(x,u,q)\mid_{u=q=1}=e^{e^x-1}(\frac{3}{4}e^{3x}+\frac{3}{2}e^{2x}-\frac{7}{4}e^x-xe^{2x}-\frac{3}{2}xe^x-\frac{1}{2}).$$
By differentiating the well known generating function $e^{e^x-1}=\sum_{n\geq0}B_n\frac{x^n}{n!}$ three times we obtain
$$e^xe^{e^x-1}=\sum_{n\geq0}B_{n+1}\frac{x^n}{n!},$$
$$e^{2x}e^{e^x-1}=\sum_{n\geq0}B_{n+2}\frac{x^n}{n!}-\sum_{n\geq0}B_{n+1}\frac{x^n}{n!}$$ and
$$e^{3x}e^{e^x-1}=\sum_{n\geq0}B_{n+3}\frac{x^n}{n!}-3\sum_{n\geq0}B_{n+2}\frac{x^n}{n!}+2\sum_{n\geq0}B_{n+1}\frac{x^n}{n!}.$$
From the above equations, we can derive that
$$xe^xe^{e^x-1}=\sum_{n\geq0}nB_n\frac{x^n}{n!}$$
and
$$xe^{2x}e^{e^x-1}=\sum_{n\geq0}nB_{n+1}\frac{x^n}{n!}-\sum_{n\geq0}nB_n\frac{x^n}{n!}.$$
Using all these facts together leads to
\begin{align*}
&\frac{d}{dq}\widetilde{P}_k(x,u,q)\mid_{u=q=1}=\sum_{n\geq0}(\frac{3}{4}B_{n+3}-\frac{3}{4}B_{n+2}-(n+\frac{7}{4})B_{n+1}-\frac{1}{2}(n+1)B_n)\frac{x^n}{n!}.
\end{align*}
Hence the total number of $\sumweightedrecord$ is given by
$$\frac{3}{4}(B_{n+3}-B_{n+2})-(n+\frac{7}{4})B_{n+1}-\frac{1}{2}(n+1)B_n.$$
\end{proof}

In order to obtain asymptotic estimate for the moment as well as limiting distribution,
we need the fact
$$B_{n+h}=B_n\frac{(n+h)!}{n!r^h}\left(1+O(\frac{\log n}{n})\right)$$
uniformly for $h=O(\log n)$, where $r$ is the positive root of $re^r=n+1$.
 For more details about
the asymptotic expansion of Bell numbers see \cite{ER}. Therefore,
 Theorem \ref{th} gives the following corollary.

\begin{corollary}
Asymptotically, the total number of $\sumweightedrecord$ taken over all set partitions
of $[n]$, is given by
$$B_n\frac{n^3}{r^3}\left(1+\frac{r}{n}\right)\left(1+O(\frac{\log n}{n})\right).$$
\end{corollary}

\vspace{0.5cm}
\textbf{Acknowledgement}. The research of the author
was supported by the Ministry of Science and Technology,
Israel.

\end{document}